\documentclass[12pt]{article}
\usepackage[latin9]{inputenc}
\usepackage{geometry}
\usepackage{verbatim}
\usepackage{amsmath}
\usepackage{amsthm}
\usepackage{amssymb}
\usepackage[unicode=true,
 bookmarks=false,
 breaklinks=false,pdfborder={0 0 1},backref=section,colorlinks=false]
 {hyperref}

\makeatletter
\usepackage{amsmath, amsthm, amssymb}
\usepackage{hyperref}
\usepackage{verbatim}
\usepackage{pdfsync}
\usepackage{fullpage}

\usepackage{sectsty}
\allsectionsfont{\sffamily}

\makeatletter
\newtheorem*{rep@theorem}{\rep@title}
\newcommand{\newreptheorem}[2]{
\newenvironment{rep#1}[1]{
 \def\rep@title{#2 \ref{##1}}
 \begin{rep@theorem}}
 {\end{rep@theorem}}}
\makeatother

\theoremstyle{plain}
\newtheorem{thm}{Theorem}[section]
\newreptheorem{thm}{Theorem}

\newreptheorem{prop}{Proposition}
\newtheorem{lem}[thm]{Lemma}
\newreptheorem{lem}{Lemma}

\newreptheorem{conjecture}{Conjecture}
\newtheorem{cor}[thm]{Corollary}
\newreptheorem{cor}{Corollary}

\newtheorem*{KernelLemma}{Kernel Lemma}
\newtheorem*{MainLemma}{Main Lemma}

\newtheorem*{BrooksTheoremAlpha}{Brooks' Theorem for Independence Number}

\theoremstyle{definition}
\newtheorem{defn}{Definition}
\theoremstyle{remark}

\newcommand{\fancy}[1]{\mathcal{#1}}
\newcommand{\IN}{\mathbb{N}}

\renewcommand{\L}{\fancy{L}}
\newcommand{\HH}{\fancy{H}}

\newcommand{\set}[1]{\left\{ #1 \right\}}

\newcommand{\card}[1]{\left|#1\right|}
\newcommand{\size}[1]{\left\Vert#1\right\Vert}
\newcommand{\ceil}[1]{\left\lceil#1\right\rceil}
\newcommand{\floor}[1]{\left\lfloor#1\right\rfloor}
\newcommand{\func}[3]{#1\colon #2 \rightarrow #3}

\newcommand{\parens}[1]{\left( #1 \right)}

\newcommand{\DefinedAs}{\mathrel{\mathop:}=}

\newcommand{\mic}{\operatorname{mic}}

\newcommand{\col}{\operatorname{col}}

\newcommand{\sm}{\smallsetminus}

\newcommand\restr[2]{{
  \left.\kern-\nulldelimiterspace 
  #1 
  \vphantom{\big|} 
  \right|_{#2} 
  }}

\makeatother

\begin{document}
\title{Extracting list colorings from large independent sets} \author{H.\,A.~Kierstead and Landon Rabern\thanks{School of Mathematical and Statistical Sciences, Arizona State University}} \date{\today}
 \maketitle

\begin{abstract}
We take an application of the Kernel Lemma by Kostochka and Yancey \cite{kostochkayancey2012ore}
to its logical conclusion. The consequence is a sort of magical way to draw conclusions
about list coloring (and online list coloring) just from the existence of an independent
set incident to many edges. We use this to prove an Ore-degree version of Brooks'
Theorem for online list-coloring. The Ore-degree of an edge $xy$ in a graph $G$ is $\theta(xy) = d_G(x) + d_G(y)$.  
The Ore-degree of $G$ is $\theta(G) = \max_{xy\in E(G)}\theta(xy)$.
We show that every graph with $\theta\ge18$ and $\omega\le\frac{\theta}{2}$
is online $\left\lfloor \frac{\theta}{2}\right\rfloor $-choosable.  In addition, we prove an upper bound for online list-coloring
triangle-free graphs: $\chi_{OL}\le\Delta+1-\lfloor\frac{1}{4}\lg(\Delta)\rfloor$.
Finally, we characterize Gallai trees as the connected graphs $G$ with no independent set incident to at least $|G|$ edges.
\end{abstract}

\section{Introduction}
In \cite{kostochkayancey2012ore}, Kostochka and Yancey applied the Kernel Lemma
to a coloring problem in a novel manner. Our Main Lemma generalizes and strengthens
their idea. The basic idea is that given an independent set that is incident to many
edges, we can find a reducible configuration. In this way, we can reduce coloring
problems to the mere existence of a large independent set. Before stating the Main Lemma  we need to introduce some notation, and define the concepts of $f$-choosable and online $f$-choosable graphs.

Let $G = (V, E)$ be a graph.  We write $|G|$ for $|V|$ and $\size{G}$ for $|E|$.  
For $A, B \subseteq V$, put $\size{A,B}_G \DefinedAs \sum_{v \in A} |N_G(v) \cap B|$.  Also, put $\size{A}_G \DefinedAs \size{G[A]}$. Note that $\size{A,B}_G = \size{A\sm B, B\sm A}_G + 2\size{A \cap B}_G$.  
In particular, $\size{A,B}_G = \size{B,A}_G$ and $\size{V, V}_G = 2\size{V}_G$. When the graph $G$ is clear from context we drop the $G$ from the notation and write $\size{A,B}$ and $\size{A}$. For $v \in V$, let $N[v] = N(v) \cup \set{v}$.

Let $G = (V,E)$ be a graph. A list assignment on $G$ is a function $L$ from $V$ to
the subsets of $\IN$. A graph $G$ is \emph{$L$-colorable} if there is $\func{\pi}{V}{\IN}$
such that $\pi(v)\in L(v)$ for each $v\in V$ and $\pi(x)\ne\pi(y)$ for each
$xy\in E$. For $\func{f}{V}{\IN}$, a list assignment $L$ is an \emph{$f$-assignment}
if $\card{L(v)}=f(v)$ for each $v\in V$. We say that $G$ is \emph{$f$-choosable}
if $G$ is $L$-colorable for every $f$-assignment $L$.

Online choosability was independently introduced by Zhu \cite{zhu2009online} and
Schauz \cite{schauz2009mr} (Schauz called it \emph{paintability}). Let $G = (V,E)$ be a
graph and $\func{f}{V}{\IN}$. We say that $G$ is \emph{online $f$-choosable}
if $|G|=0$ or both $f(v)\ge 1$ for all $v\in V$ and for every $S\subseteq V$ there is an
independent set $I\subseteq S$ such that $G-I$ is online $f'$-choosable where
$f'(v)\DefinedAs f(v)$ for $v\in V\sm S$ and $f'(v)\DefinedAs f(v)-1$ for $v\in S\sm I$.  

Observe that if $G$ is online $f$-choosable then it is $f$-choosable: For any $f$-list assignment $L$ on  $G$, let $v\in V(G)$, $\alpha\in L(v)$, and $S=\{x\in V(G): \alpha\in L(x)\}$. As $G$ is online $f$-choosable, there is an independent set $I\subseteq S$ with $G-I$ online $f'$-choosable. Color each vertex of $I$ with $\alpha$, and let $L'$ be the $f'$-list assignment on $G-I$ obtained by removing $\alpha$ from every list. Arguing inductively (with a trivial basis) yields an $L'$-coloring of $G-I$, and an $L$-coloring of $G$. 

When $f(v) \DefinedAs k-1$ for all $v \in V(G)$, we say that $G$ is \emph{online $k$-list-critical} if $G$ is not online $f$-choosable, 
but every proper subgraph $H$ of $G$ is online $\restr{f}{V(H)}$-choosable.

\begin{MainLemma} Let $G = (V,E)$ be a nonempty graph and $\func{f}{V}{\IN}$ with $f(v)\le d_{G}(v)+1$
for all $v\in V$. If there is an independent $A\subseteq V$ such that 
\[
\size{A,V}\ge\sum_{v\in V}d_{G}(v)+1-f(v),
\]
then $G$ has a nonempty induced subgraph $H$ that is (online) $f_{H}$-choosable
where $f_{H}(v)\DefinedAs f(v)+d_{H}(v)-d_{G}(v)$ for $v\in V(H)$. \end{MainLemma}

As a simple first application, we show that Brooks' Theorem can be derived from
the simple bound on the independence number it implies. In particular, the proof shows that Brooks' Theorem for list coloring \cite{vizing1976} 
and online list coloring \cite{Hladky} follow from Brooks' Theorem for ordinary coloring.

\begin{BrooksTheoremAlpha} If $G$ is a graph with $\Delta\DefinedAs \Delta(G)\ge3$ and
$K_{\Delta+1}\not\subseteq G$, then $\alpha(G)\ge\frac{|G|}{\Delta}$. 
\end{BrooksTheoremAlpha}

To prove Brooks' Theorem, consider a minimal counterexample $G$. By minimality $G$
is regular: if $d(x)<\Delta$ then $G-x$ has a $\Delta$-coloring, and some color
is used on no neighbor of $x$, a contradiction. Hence for any maximum independent
set $A$ in $G$, Brooks' Theorem for Independence Number gives $\size{A,V(G)}\ge|A|\Delta\ge|G|$.
Applying the Main Lemma with $f(v)\DefinedAs d_{G}(v)$ gives a nonempty induced
subgraph $H$ of $G$ that is (online) $d_{H}$-choosable. So, after $\Delta$-coloring
$G-H$ by minimality of $G$, we can finish the coloring on $H$, a contradiction.

A bound like Brooks' Theorem in terms of the Ore-degree was given by Kierstead and
Kostochka \cite{kierstead2009ore} and subsequently the required lower bound on $\Delta$
was improved in \cite{rabern2010a, krs_one, rabern2012partitioning}. For example,
we have the following.

\begin{defn} The \emph{Ore-degree} of an edge $xy$ in a graph $G$ is $\theta(xy)\DefinedAs d(x)+d(y)$.
The \emph{Ore-degree} of a graph $G$ is $\theta(G)\DefinedAs\max_{xy\in E(G)}\theta(xy)$.
\end{defn}

\begin{thm}\label{RegularOre} Every graph with $\theta\ge10$ and $\omega\leq\frac{\theta}{2}$
is $\floor{\frac{\theta}{2}}$-colorable. 
\end{thm}

Another method for achieving the tightest of these results on Ore-degree was given
by Kostochka and Yancey \cite{kostochkayancey2012ore}. Their proof combined their
new lower bound on the number of edges in a color critical graph together with their
list coloring lemma derived via the kernel lemma. The Main Lemma improves this latter
lemma and, in a similar way, we use it in combination with our lower bound on the
number of edges in online list-critical graphs \cite{OreVizing} to prove an Ore-degree
version of Brooks' Theorem for online list coloring.

Now we introduce the key graph theoretic parameter for this paper. 

\begin{defn} The \emph{maximum independent cover number }of a graph $G$ 
is the maximum $\mic(G)$ of $\sum_{v\in I}d_{G}(v)~(=\|I,V(G)\|)$ over all independent sets $I$
of $G$. A set $I$ that witnesses this maximum is said to be \emph{optimal}. 
\end{defn}

We work in terms of a class of graphs that is more general than the class of online $k$-list-critical graphs.  
Basically, we want graphs $G$ that have no induced subgraph $H$ such that every online $(k-1)$-list-coloring of $G-H$ can be extended to $H$.
We can replace $k-1$ with $\delta(G)$ and still get a generalization of online $k$-list-critical since an online $k$-list-critical graph has minimum degree at least $k-1$.
Doing so, we get the following graph class that does not depend on $k$, but still does everything we need.

\begin{defn} A graph $G$ is \emph{OC-reducible} to $H$ if $H$ is a nonempty induced
subgraph of $G$ which is online $f_{H}$-choosable where $f_{H}(v)\DefinedAs\delta(G)+d_{H}(v)-d_{G}(v)$
for all $v\in V(H)$. If $G$ is not OC-reducible to any nonempty induced subgraph,
then it is \emph{OC-irreducible}. \end{defn}

The Main Lemma can be used to give another lower bound on the number of edges in
a critical graph $G$. Viewed differently, it gives an upper bound on $\mic(G)$.

\begin{repthm}{ConsantListMicStrength} Every OC-irreducible graph $G$ satisfies
 $\mic(G)\leq2\size{G}-(\delta(G)-1)\card{G}-1$. \end{repthm}

\noindent This quickly gives the aforementioned Ore degree version of Brooks' Theorem
for list coloring.

\begin{repthm}{OurListOre} Every graph with $\theta\ge18$ and $\omega\leq\frac{\theta}{2}$
is $\floor{\frac{\theta}{2}}$-choosable. 
\end{repthm}

\noindent %
Note that using Kostochka and Stiebitz's lower bound on the number of edges in a
list critical graph \cite{kostochkastiebitzedgesincriticalgraph} gives a weakened version of Theorem \ref{OurListOre}
with $\theta\ge54$ instead of $\theta\ge 18$. Similarly, we get the online version.

\begin{repthm}{OurListOnlineOre} Every graph with $\theta\ge18$ and $\omega\leq\frac{\theta}{2}$
is online $\floor{\frac{\theta}{2}}$-choosable. \end{repthm}

We expect that Theorems \ref{OurListOnlineOre} and \ref{OurListOre} actually hold
for $\theta\ge10$. In the regular coloring case, it was shown in \cite{krs_one}
that the only exception when $\theta\ge8$ is the graph $O_{5}$; where $O_n$ is the graph formed from the disjoint union of $K_n - xy$ and
$K_{n-1}$ by joining $\floor{\frac{n-1}{2}}$ vertices of the $K_{n-1}$ to $x$
and the other $\ceil{\frac{n-1}{2}}$ vertices of the $K_{n-1}$ to $y$.
Again, the expectation is that the same result will hold for Theorems \ref{OurListOnlineOre}
and \ref{OurListOre}.

A simple probabilistic argument gives a reasonable bound on $\mic(G)$ for
triangle-free graphs and we get the following. 

\begin{repcor}{tricolor} 
Triangle-free graphs are online $\parens{\Delta+1-\floor{\frac{1}{4}\lg(\Delta)}}$-choosable.
\end{repcor}

In the final section, we characterize Gallai trees as the connected graphs $G$ that have $\mic(G) = |G| - 1$, and some further characterizations.
\section{Proving the Main Lemma}

A \emph{kernel} in a digraph $D$ is an independent set $I\subseteq V(D)$ such that
each vertex in $V(D)\sm I$ has an edge into $I$. A digraph in which every induced
subdigraph has a kernel is \emph{kernel-perfect}.

\begin{KernelLemma} Let $G = (V,E)$ be a graph and $\func{f}{V}{\IN}$. If $G$ has
a kernel-perfect orientation such that $f(v)\ge d^{+}(v)+1$ for each $v\in V$,
then $G$ is online $f$-choosable. \end{KernelLemma}

\begin{lem}[Kostochka and Yancey \cite{kostochkayancey2012ore}]\label{KernelPerfect}
Let $A$ be an independent set in a graph $G$ and let $B\DefinedAs V(G)\sm A$. Any
digraph $D$ created from $G$ by replacing each edge in $G[B]$ by a pair of opposite
arcs and orienting the edges between $A$ and $B$ arbitrarily is kernel-perfect.
\end{lem} 
\begin{proof} Let $G$ be a minimum counterexample, and let $D$ be a
digraph created from $G$ that is not kernel-perfect. To get a contradiction it suffices
to construct a kernel in $D$, since each subdigraph has a kernel by minimality.
Either $A$ is a kernel or there is some $v\in B$ which has no outneighbors in $A$.
In the latter case, each neighbor of $v$ in $G$ has an inedge to $v$, so a kernel
in $D-v-N(v)$ together with $v$ is a kernel in $D$. 
\end{proof}

The following lemma is folklore and can be derived from Hall's theorem by vertex
splitting. It also follows by taking an arbitrary orientation and repeatedly reversing
paths if doing so gets a gain.

\begin{lem}\label{InOrientations} Let $G = (V,E)$ be a graph and $\func{g}{V}{\IN}$.
Then $G$ has an orientation such that $d^{-}(v)\ge g(v)$ for all $v\in V$ if and only if
for every $X \subseteq V$, we have

\[
\size{X}+\size{X,V\sm X}\ge\sum_{v\in X} g(v).
\]
\end{lem}

For independent $A\subseteq V(G)$, we write $G_{A}$ for the bipartite subgraph
$G-E(G-A)$ of $G$. Thus $G$ is the edge-disjoint union of $G_{A}$ and $G[V(G)\sm A]$.

\begin{lem}\label{MicStrength} Let $G = (V,E)$ be a graph and $\func{f}{V}{\IN}$ with
$f(v)\le d_{G}(v)+1$ for all $v\in V$. If there is an independent $A\subseteq V$
such that each $X \subseteq V(G_{A})$ satisfies 

\[
\size{X}_{G_{A}}+\size{X,V(G_{A})\sm X}_{G_{A}}\ge\sum_{v\in X}(d_{G}(v)+1-f(v)).
\]

\noindent then $G$ is online $f$-choosable. 
\end{lem}

\noindent \begin{proof} 
Applying Lemma \ref{InOrientations} on $G_{A}$ with $g(v)\DefinedAs d_{G}(v)+1-f(v)$
for all $v\in V(G_{A})$ gives an orientation of $G_{A}$ where $d^{-}(v)\ge d_{G}(v)+1-f(v)$
for each $v\in V(G_{A})$. Make an orientation $D$ of $G$ by using this orientation
of $G_{A}$ for the edges between $A$ and $V(G)\sm A$ and replacing each edge in $G-A$
by a pair of opposite arcs. For $v\in V(D)$, where $d_{G-A}(v)=0$ if $v\in A$,
\[
d^{+}(v)\le d_{G-A}(v)+d_{G_{A}}(v)-(d_{G}(v)+1-f(v))=f(v)-1,
\]
so $f(v)\ge d^{+}(v)+1$. By Lemma \ref{KernelPerfect}, $D$ is kernel-perfect,
so the Kernel Lemma implies $G$ is online $f$-choosable. \end{proof}

\begin{proof}[Proof of Main Lemma] Let $A\subseteq V$ be an independent set
with 
\[
\size{A,V}\ge\sum_{v\in V}\parens{d_{G}(v)+1-f(v)}.
\]
Choose a nonempty induced subgraph $H$ of $G$ with $\size{H_{A}}\ge\sum_{v\in V(H)}\parens{d_{H}(v)+1-f_{H}(v)}$
minimizing $\card{H}$ (we can make this choice since $G$ is a such a subgraph).
Suppose $H$ is not online $f_{H}$-choosable. Then, by Lemma \ref{MicStrength},
we have $X \subseteq V(H_{A})$ with
\[\size{X}_{H_{A}}+\size{X,V(H_{A})\sm X}_{H_{A}} < \sum_{v\in X}(d_{H}(v)+1-f_H(v)).\]
Now $X \ne V(H)$ by our assumption on $\size{H_{A}}$, hence $Z\DefinedAs H-X$ is a
nonempty induced subgraph of $G$ with 
\begin{align*}
\size{Z_{A}} & =\size{H_{A}}-\size{X}_{H_{A}}-\size{X,V(H_{A})\sm X}_{H_{A}}\\
 & >\sum_{v\in V(H)}\parens{d_{H}(v)+1-f_{H}(v)}-\sum_{v\in X}\parens{d_{H}(v)+1-f_{H}(v)}\\
 & =\sum_{v\in V(Z)}\parens{d_{Z}(v)+1-f_{Z}(v)},
\end{align*}
contradicting the minimality of $\card{H}$. \end{proof}

As a special case we get:

\begin{thm}\label{ConsantListMicStrength} Every OC-irreducible graph $G = (V,E)$ satisfies
 $$\mic(G)\leq2\size{G}-(\delta(G)-1)\card{G}-1.$$\end{thm}

\begin{proof}As $G$ is OC-irreducible, it has no proper, induced, online $f_{H}$-choosable
subgraph $H$, where $f_{H}(v)=\delta(G)+d_{H}(v)-d_{G}(v)$. Let $A \subseteq V$ be an optimal
set. By the Main Lemma 
\[
\mic(G)=\size{A,V} <\sum_{v\in V}(d_{G}(v)+1-\delta(G))=2\size{G} +|G|(1-\delta(G)).\qedhere
\]
 \end{proof}
 

\section{Ore version of Brooks' Theorem for (online) list coloring}
For a graph $G$, let $\HH(G)$ be the subgraph of $G$ induced on the vertices of
degree greater than $\delta(G)$ and $\L(G)$ the subgraph of $G$ induced on the
vertices of degree $\delta(G)$.

\begin{lem}\label{OrePrecursor1} All OC-irreducible graphs $G$ with $\HH(G)$
 edgeless and $\Delta(G)=\delta(G)+1$ satisfy
\begin{enumerate}
 \item $\mic(G)<\card{\HH(G)}+\card{G}$, and
 \item $\parens{\delta-1}\card{\HH(G)}< \card{\L(G)}$, and $2\size{G}<\parens{\delta(G)+\frac{1}{\delta(G)}}\card{G}$.
 \end{enumerate}
\end{lem} 
\begin{proof} Put $\delta\DefinedAs\delta(G)$ and $\Delta\DefinedAs\Delta(G)$.
 As  $\Delta=\delta+1$, $2\size{G}=\delta\card{G}+\card{\HH(G)}$. As $\HH(G)$ is edgeless, and using Theorem~\ref{ConsantListMicStrength},
   $$(\delta+1)\card{\HH(G)}\le \mic(G)<2\size{G}+\card{G}(1-\delta)
   =\card{\HH(G)}+\card{G}.$$ 
Thus (1) holds, and  $\delta\card{\HH(G)}<\card{G}$, $\parens{\delta-1}\card{\HH(G)}< \card{\L(G)}$, and  $2\size{G}<(\delta+1/\delta)\card{G}$.
 \end{proof}

To break up our computations we reformulate Lemma \ref{OrePrecursor1} as an upper
bound on $\sigma$ where 
\[
\sigma(G)\DefinedAs\parens{\delta(G)-1+\frac{2}{\delta(G)}}\card{\L(G)}-2\size{\L(G)}.
\]

\begin{lem}\label{OrePrecursor2} If $G$ is an OC-irreducible graph such that $\HH(G)$
is edgeless and $\Delta(G)=\delta(G)+1$, then $\sigma(G)<\parens{4-\frac{2}{\delta(G)}}\card{\HH(G)}$.
\end{lem} 

\begin{proof} Put $\delta\DefinedAs\delta(G)$. As $(1+\delta)\card{\HH(G)}=\size{\HH(G),\L(G)}=\delta\card{\L(G)}-2\size{\L(G)}$,
\begin{align*}
\sigma(G) & =\parens{\delta-1+\frac{2}{\delta}}\card{\L(G)}+(1+\delta)\card{\HH(G)}-\delta\card{\L(G)}\\
 & =\parens{\frac{2}{\delta}-1}\card{\L(G)}+(1+\delta)\card{\HH(G)}\\
 & <\parens{\frac{2}{\delta}-1}\parens{\delta-1}\card{\HH(G)}+(1+\delta)\card{\HH(G)} & (\textrm{Lemma~\ref{OrePrecursor1}.2)}\\
 & =\parens{4-\frac{2}{\delta}}\card{\HH(G)}.& & \qedhere
\end{align*}
\end{proof}


We need the following bound from \cite{OreVizing}. Put $\alpha_{k}\DefinedAs\frac{1}{2}-\frac{1}{(k-1)(k-2)}$
and let $c(G)$ be the number of components in $G$. \begin{cor}\label{SigmaCorollary}
If $G$ is an OC-irreducible graph with $\delta(G)\ge6$ and $\omega(G)\le\delta(G)$
such that $\HH(G)$ is edgeless, then $\sigma(G)\ge(\delta(G)-2)\alpha_{\delta(G)+1}\card{\HH(G)}+2(1-\alpha_{\delta(G)+1})c(\L(G))$.
\end{cor}

By combining Lemma \ref{OrePrecursor2} with Corollary \ref{SigmaCorollary} we can
prove the Ore version of Brooks' Theorem for online list coloring for $\Delta\ge11$.
With a bit more work we will improve this to $\Delta\ge10$. First, we can squeeze a bit more out of Theorem~\ref{ConsantListMicStrength} by considering independent
sets of low vertices that have no high neighbors. Such sets can be added to $V(\HH(G))$
to get a cut with more edges. To apply this idea we need the following counting lemma. 
For a graph $G$ and $t\in\IN$,  let $V_t=\{v\in V(G):d(v)=t\}$, $G_t=G[V_t]$ and $\beta_t=\alpha(G_t)$.

To apply this idea, we need to better understand the structure of $\L(G)$ when $G$ is OC-irreducible. 
A \emph{Gallai tree} is a graph such that each block is a clique or odd cycle.  A \emph{Gallai forest} is a disjoint union of Gallai trees.
A classical result of Gallai \cite{gallai1963kritische} says that if $G$ is a $k$-critical graph, then $\L(G)$ is a Gallai forest.  
Borodin \cite{borodin1977criterion} and independently Erd\H{o}s, Rubin and Taylor \cite{erdos1979choosability} generalized Gallai's result to show that
a connected graph $H$ is $f$-choosable where $f(v) = d_H(v)$ for all $v \in V(H)$ if and only if $H$ is not a Gallai tree.  
In \cite{Hladky}, Hladk{\`y}, Kr{\'a}l and Schauz extended this result to online $f$-choosability.   

\begin{lem}[Hladk{\`y}, Kr{\'a}l and Schauz]\label{OnlineDegreeChoosable}
	A connected graph $H$ is online $f$-choosable where $f(v) = d_H(v)$ for all $v \in V(H)$ if and only if $H$ is not a Gallai tree.
\end{lem}

In Section 6, we give another proof of Lemma~\ref{OnlineDegreeChoosable} using our Main Lemma. We need Gallai's result for OC-irreducible graphs.

\begin{lem}\label{GallaiForestOC}
	If $G$ is OC-irreducible, then $\L(G)$ is a Gallai forest.
\end{lem}
\begin{proof}
	Let $G$ be an OC-irreducible graph and $H$ a component of $\L(G)$.  Since $G$ is OC-irreducible, 
	$H$ is not online $f_H$-choosable where $f_H(v) \DefinedAs \delta(G) + d_H(v) - d_G(v)$ for all $v \in V(H)$.  
	But $d_G(v) = \delta(G)$ for $v \in V(H)$, so $f_H(v) = d_H(v)$ for all $v \in V(H)$. By Lemma~\ref{OnlineDegreeChoosable}, $H$ is a Gallai tree.
\end{proof}

\begin{lem}\label{GallaiTreeCount} Fix $k\ge6$. Let $G$ be a Gallai forest with
maximum degree at most $k-1$ not containing $K_{k}$. We have the following inequality:

\[
(k-1)\beta_{k-1}(G)+\sum_{v\in V(G)}k-1-d(v)\ge\frac{2(k-3)}{k-2}\card{G}-\frac{(k-1)(k-4)}{k-2}c(G).
\]
\end{lem} \begin{proof} It will suffice to prove that for any Gallai tree $T$
with maximum degree at most $k-1$ we have: 
\[
(k-1)\beta_{k-1}(T)+\sum_{v\in V(T)}\parens{k-1-d(v)}\ge\frac{2(k-3)}{k-2}\card{T}-\frac{(k-1)(k-4)}{k-2}.
\]

Suppose not and choose a counterexample $T$ minimizing $\card{T}$. First, if $T$
has only one block it is easy to see that the inequality is satisfied. Let $B$ be
an endblock of $T$ and say $x$ is the cutvertex in $B$. Suppose $\chi(B)\leq k-3$.
Put $T^*\DefinedAs T-(B-x)$. By minimality of $\card{T}$, $T^*$ satisfies the inequality.
Adding $B-x$ back increases the left side  by $(k-\chi(B))\card{B-x}-\size{x,B-x}\ge2(\card{B}-1)$,
but only increases the right side by $\frac{2(k-3)}{k-2}(\card{B}-1)$, so
$T$ is not a counterexample, a contradiction. Thus $B$ is either $K_{k-2}$ or $K_{k-1}$.

Consider $T'\DefinedAs T-B$. First suppose
$d_{T}(x)=k-1$. Note that none of $x$'s neighbors in $T'$ have degree $k-1$ in
$T'$ and thus are in no maximum independent set of degree $k-1$ vertices in $T'$.
Therefore, we can add $x$ to any such independent set, giving $\beta_{k-1}(T)>\beta_{k-1}(T')$.
Hence, after applying minimality to $T'$, we see that adding back $B$ increases
the left side by $k-1+\card{B-x}-\size{x,T'}=2k-4$ if $B$ is $K_{k-1}$ and by $k-1+2\card{B-x}-\size{x,T'}=3k-9$ if $B$
is $K_{k-2}$. Since the right side increases by only $\frac{2(k-3)}{k-2}\card{B}$
in both cases, $T$ satisfies the inequality, a contradiction.

Otherwise, $d_{T}(x)\le k-2$. So $B$ is $K_{k-2}$ and $d_{T}(x)=k-2$.  Now  adding 
$B$ back, increases the left side  by $2(k-3)+1-1$ and increases the right side 
by only $2(k-3)$, so again $T$ satisfies the inequality, a contradiction. \end{proof}

\begin{lem}\label{OrePrecursor3} If $G$ is an OC-irreducible graph such that $\HH(G)$
is edgeless, $\Delta(G)=\delta(G)+1\ge7$ and $K_{\Delta(G)}\not\subseteq G$, then
$\card{\HH(G)}<\frac{\delta(G)(\delta(G)-3)}{(\delta(G)-1)(\delta(G)-5)}c(\L(G))$.
\end{lem}

 \begin{proof} Put $\delta\DefinedAs\delta(G)$. By Lemma~\ref{OrePrecursor1}.1
we have
 \[\mic(G)<\card{\L(G)}+2\card{\HH(G)}<\card{\L(G)}+\frac{2}{\delta-1}\card{\L(G)}=\frac{\delta+1}{\delta-1}\card{\L(G)}.\]
 
 By Lemma~\ref{GallaiForestOC}, $\L(G)$ is a Gallai forest. 
 Pick $I\subseteq V_{k-1}$ with 
$\card{I}=\beta_{k-1}(G)$, and set $J=I\cup \HH(G)$.  
 As  $J$ is independent, applying Lemma \ref{GallaiTreeCount} to the Gallai forest $\L(G)$ gives
 \begin{align*}
 \mic(G)& \ge\size{J,V(G)} \ge\delta\beta_{k-1}(\L(G))+\sum_{v\in V(\L(G))}\delta-d_{\L(G)}(v)\\
 & \ge2\frac{\delta-2}{\delta-1}\card{\L(G)}-\frac{\delta(\delta-3)}{\delta-1}c(\L(G)).
 \end{align*}
 By Lemma~\ref{OrePrecursor1}.2, $\card{\L(G)}>(\delta-1)\card{\HH(G)}$. Combining this with the inequalities above proves the
lemma. \end{proof}

\begin{lem}\label{EdgelessEuler} Every OC-irreducible graph $G$ with $\delta(G)+1=\Delta(G)\ge10$
such that $\HH(G)$ is edgeless contains $K_{\Delta(G)}$. \end{lem} \begin{proof}
Suppose not and let $G$ be a counterexample. Put $\delta\DefinedAs\delta(G)$. Then
\begin{align*}
\parens{4-\frac{2}{\delta}}\card{\HH(G)}& >\sigma(G) & \textrm{(Lemma~\ref{OrePrecursor2})}\\ & \ge 
(\delta-2)\alpha_{\delta+1}\card{\HH(G)}+2(1-\alpha_{\delta+1})c(\L(G)) & \textrm{(Corollary \ref{SigmaCorollary})}\\ & \ge
\card{\HH(G)}\parens{(\delta-2)\alpha_{\delta+1}+2(1-\alpha_{\delta+1})\frac{(\delta-1)(\delta-5)}{\delta(\delta-3)}} &\textrm{(Lemma \ref{OrePrecursor3})} \\ 
 4-\frac{2}{\delta}& >\alpha_{\delta+1}(\delta-2)+2(1-\alpha_{\delta+1})\frac{(\delta-1)(\delta-5)}{\delta(\delta-3)}. 
\end{align*}
Thus $\delta\leq8$, a contradiction. \end{proof}

The following lemma from \cite{schauz2009mr} allows us to infer online list colorability of the whole from
online list colorability of parts.

\begin{lem}\label{CutLemma} Let $G$ be a graph and $\func{f}{V(G)}{\IN}$. If
	$H$ is an induced subgraph of $G$ such that $G-H$ is online $\restr{f}{V(G-H)}$-choosable
	and $H$ is online $f_{H}$-choosable where $f_{H}(v)\DefinedAs f(v)+d_{H}(v)-d_{G}(v)$,
	then $G$ is online $f$-choosable. \end{lem}

\begin{thm}\label{EdgelessOnlineEuler} If $G$ is a graph with $\Delta(G)\ge10$
not containing $K_{\Delta(G)}$ such that $\HH(G)$ is edgeless, then $G$ is online
$(\Delta(G)-1)$-choosable. \end{thm} \begin{proof} Suppose not and choose a counterexample
$G$ minimizing $\card{G}$. Then $G$ is online $f$-critical where $f(v)\DefinedAs\Delta(G)-1$
for all $v\in V(G)$. Hence $\delta(G)\ge\Delta(G)-1$ and we may apply Lemma \ref{EdgelessEuler}
to get a nonempty induced subgraph $H$ of $G$ that is online $f_{H}$-choosable
where $f_{H}(v)\DefinedAs\Delta(G)-1+d_{H}(v)-d_{G}(v)$ for all $v\in V(H)$. But
then applying Lemma \ref{CutLemma} shows that $G$ is $(\Delta(G)-1)$-choosable,
a contradiction. \end{proof}

Combining Theorem \ref{EdgelessOnlineEuler} with the following version of Brooks'
Theorem for online list coloring (first proved in \cite{Hladky}) we get Theorem
\ref{OurListOnlineOre}.

\begin{lem}\label{BrooksOnline} Every graph with $\Delta\ge3$ not containing $K_{\Delta+1}$
is online $\Delta$-choosable. \end{lem}

\begin{thm}\label{OurListOnlineOre} Every graph with $\theta\ge18$ and $\omega\leq\frac{\theta}{2}$
is online $\floor{\frac{\theta}{2}}$-choosable. \end{thm} \begin{proof} Suppose
not and choose a counterexample $G$ minimizing $\card{G}$. Put $k\DefinedAs\floor{\frac{\theta(G)}{2}}$.
Then $G$ is online $f$-critical where $f(v)\DefinedAs k$ for all $v\in V(G)$.
Hence $\delta(G)\ge k$ and thus $\Delta(G)\leq k+1$. If $\Delta(G)=k$, then the
theorem follows from Lemma \ref{BrooksOnline}. Hence we must have $\Delta(G)=k+1$.
Therefore $\HH(G)$ is edgeless, $\Delta(G)\ge10$ and $\omega(G)\leq\Delta(G)-1$.
Applying Theorem \ref{EdgelessOnlineEuler} shows that $G$ is online $(\Delta(G)-1)$-choosable,
a contradiction. \end{proof}

The same result for list coloring is an immediate consequence.

\begin{thm}\label{OurListOre} Every graph with $\theta\ge18$ and $\omega\leq\frac{\theta}{2}$
is $\floor{\frac{\theta}{2}}$-choosable. \end{thm}

\section{Ore Brooks for maximum degree four}

Kostochka and Yancey's bound \cite{kostochkayancey2012ore} shows that if $G$ is
$4$-critical, then $\size{G}\ge\ceil{\frac{5\card{G}-2}{3}}$. If we try to analyze
$4$-critical graphs with edgeless high vertex subgraphs by putting this lower bound
on the number of edges together with the results on orientations and list coloring
obtained in \cite{kostochkayancey2012ore}, the bounds miss each other. Using the
improved bound from Lemma \ref{OrePrecursor1} we get an exact bound on the number
of edges in such a graph.

\begin{lem}\label{EdgesIn4Critical} For a critical graph $G$ with $\Delta(G)\leq\chi(G)=4$
such that $\HH(G)$ is edgeless we have $\size{G}=\ceil{\frac{5\card{G}-2}{3}}$
and $\card{G}$ is not a multiple of $3$. \end{lem} \begin{proof} Since $G$ is
$4$-critical, applying Lemma \ref{OrePrecursor1}.1 gives $2\size{G}<\parens{3+\frac{1}{3}}\card{G}=\frac{10}{3}\card{G}$.
By Kostochka and Yancey's bound we have $\ceil{\frac{5\card{G}-2}{3}}\leq\size{G}<\frac{5}{3}\card{G}$.
Hence $\size{G}=\ceil{\frac{5\card{G}-2}{3}}$ and $\card{G}$ is not a multiple
of $3$. \end{proof}

It is easy to see that contracting a diamond in a critical graph $G$ with $\Delta(G)\leq\chi(G)=4$
such that $\HH(G)$ is edgeless gives another such graph. So, the following characterization
of these graphs is natural. Recently, Postle \cite{postle} proved this using an
extension of the potential method of Kostochka and Yancey.

\begin{thm}[Postle \cite{postle}]\label{ContractionConjecture} Every critical
graph $G$ with $\Delta(G)\leq\chi(G)=4$ such that $\HH(G)$ is edgeless, except
$K_{4}$, has an induced diamond. In particular, any such $G$ can be reduced to
$K_{4}$ by a sequence of diamond contractions. \end{thm}

\section{Online choosability of triangle-free graphs}

We write $\lg(x)$ for the base $2$ logarithm of $x$. We can get a reasonably good
lower bound on $\mic(G)$ for triangle-free graphs using a simple probabilistic technique
of Shearer and its modification by Alon (see \cite{alon2004probabilistic}). 

\begin{lem}\label{triangle-free-mic} Every triangle-free graph $G=(V,E)$ satisfies
$\mic(G)\ge\frac{1}{4}\sum_{v\in V}\lg(d(v))$. \end{lem}%

\begin{proof} Let $W$ be a random independent set in $G$ chosen uniformly from
all independent sets in $G$. It suffices to show that $E(\left\Vert W,V\right\Vert )\geq\frac{1}{4}\sum_{v\in V}\lg(d(v))$.
For each $v\in V$ put 
\[
X_{v}\DefinedAs\begin{cases}
d(v) & \mbox{if \ensuremath{v\in W}}\\
\left\Vert v,W\right\Vert  & \mbox{if \ensuremath{v\notin W}}
\end{cases}.
\]
Then $\left\Vert W,V\right\Vert =\frac{1}{2}\sum_{v\in V}X_{v}$. By linearity of
expectation it suffices to prove 
\begin{equation}
E(X_{v})\ge\frac{1}{2}\lg(d(v)).\label{L1cl-eq}
\end{equation}

To prove \eqref{L1cl-eq}, let $H:=G[V\sm N[v]]$, fix an independent set $S$
in $H$, and set $X:=N(v)\sm N(S)$. Put $x\DefinedAs|X|$. It suffices to
prove that all such $S$ satisfy 
\begin{equation}
E\parens{X_{v}\mid W\cap V(H)=S}\ge\frac{\lg(d(v))}{2}.\label{L1cl2-eq}
\end{equation}

Suppose \eqref{L1cl2-eq} fails for $S$. As $G$ is triangle-free, the independent
sets $W$ with $W\cap V(H)=S$ are exactly $S\cup\{v\}$ and $S\cup X_{0}$ where
$X_{0}\subseteq X$. Thus 
\[
\frac{\lg(d(v))}{2}>E\parens{X_{v}\mid W\cap V(H)=S}=\frac{x2^{x-1}+d(v)}{2^{x}+1}.
\]
 So $2^{x}\lg(d(v))+\lg(d(v))>x2^{x}+2d(v)$. Putting $t\DefinedAs\lg(d(v))-x$ and
rearranging yields
\[
2^{x}t=2^{x}\parens{\lg(d(v))-x}>2d(v)-\lg(d(v))>d(v).
\]
Now we have the contradiction 
\[
\frac{t}{2^{t}}=\frac{2^{x}t}{d(v)}>1.\qedhere
\]
\end{proof}

\begin{thm}\label{triangle-free-chooooser} If $G$ is a triangle-free graph and
$\func{f}{V(G)}{\IN}$ by $f(v)\DefinedAs d_{G}(v)+1-\floor{\frac{1}{4}\lg(d_{G}(v))}$,
then $G$ has a nonempty induced subgraph $H$ that is online $f_{H}$-choosable
where $f_{H}(v)\DefinedAs f(v)+d_{H}(v)-d_{G}(v)$ for $v\in V(H)$. \end{thm} \begin{proof}
Immediate upon applying Main Lemma to $G$ since 
\[
\sum_{v\in V(G)}d_{G}(v)+1-f(v)=\sum_{v\in V(G)}\floor{\frac{1}{4}\lg(d_{G}(v))}\le\mic(G).
\]
\end{proof}

\begin{cor}\label{tricolor} If $G$ is a triangle-free graph with $\Delta(G)\le t$
for some $t\in\IN$, then $G$ is online $\parens{t+1-\floor{\frac{1}{4}\lg(t)}}$-choosable.
\end{cor} \begin{proof} Suppose not and choose a counterexample $G$ and $t\in\IN$
so as to minimize $\card{G}$. Put $f(v)\DefinedAs d_{G}(v)+1-\floor{\frac{1}{4}\lg(d_{G}(v))}$.
By Theorem \ref{triangle-free-chooooser}, $G$ has a nonempty induced subgraph $H$
that is online $f_{H}$-choosable where $f_{H}(v)\DefinedAs f(v)+d_{H}(v)-d_{G}(v)$
for $v\in V(H)$. Since $t+1-\floor{\frac{1}{4}\lg(t)}\ge d_{G}(v)+1-\floor{\frac{1}{4}\lg(d_{G}(v))}$
for all $v\in V(G)$, we have that $H$ is $g(v)$-choosable where $g(v)\DefinedAs t+1-\floor{\frac{1}{4}\lg(t)}+d_{H}(v)-d_{G}(v)$.
Now applying minimality of $\card{G}$ and Lemma \ref{CutLemma} gives a contradiction.
\end{proof}

The best, known bounds for the chromatic number of triangle-free graphs are Kostochka's
upper bound of $\frac{2}{3}\Delta+2$ in \cite{kostochka1982modification} (see \cite{rabern2010destroying}
for a proof in English) for small $\Delta$ and Johansson's upper bound of $\frac{9\Delta}{\ln(\Delta)}$
for large $\Delta$. Johansson's proof also works for list coloring, but not for
online list coloring. To the best of our knowledge Corollary \ref{tricolor} is the
best, known-upper bound for online list colorings of triangle-free graphs. Additionally,
Corollary \ref{tricolor} improves on Johansson's bound for list coloring for $\Delta\le8000$.
The bound can surely be improved by a more complicated computation of $\mic(G)$,
but not beyond around $\Delta+1-\floor{2\ln(\Delta)}$ via this method as can be
seen by examples of triangle-free graphs with independence number near $\frac{2\ln(\Delta)}{\Delta}n$.

\section{Gallai Forests}

Recall that a Gallai forest is a graph such that each block is a clique or odd cycle.
In this section we add to the many characterizations of Gallai forests. %

\subsection{Graphs with minimum mic}

\begin{lem}\label{mmic}Let $G = (V,E)$ be a connected graph with a connected induced subgraph
	$H$. Then $\mic(G)\ge\mic(H)+|G-H|$. In particular, $\mic(G)\ge|G|-1$. \end{lem}
\begin{proof}
	Argue by induction on $|G-H|$. Plainly, $|G-H| > 0$. Let $v\notin V(H)$ be a noncutvertex
	in $G/H$. By induction $\mic(G-v)\ge\mic(H)+|G-H-v|$. Let $A$ be an optimal set
	in $G-v$. Put $A'\DefinedAs A\cup\set{v}$ if $\size{v,A}=0$; else
	put $A'\DefinedAs A$. Then $A'$ is independent in $G$ and $\size{A',V} \ge\mic(H)+|G-H|$.
	Letting $|H|=1$ yields $\mic(G)\ge |G|-1$.
\end{proof}
\begin{lem} \label{mGt-lem}Every Gallai tree $G=(V,E)$ satisfies $\mic(G)=|G|-1$.
\end{lem}

\begin{proof}Argue by induction on the number of blocks. If $G$ has only one block
	then $G$ is a clique or an odd cycle, and the lemma holds. Else let $B$ be an endblock
	with cutvertex $x$. Then $G'\DefinedAs G-(B-x)$ is a Gallai tree with fewer blocks. Let
	$A=I\cup J$ be an optimal set in $G$, where $I\subseteq V(B)$ and $J\subseteq V(G')$.
	Lemma \ref{mmic} and induction yield the equality 
	\[
	|G|-1\leq\mic(G)=\size{A,V}_G = \size{I,V(B)}_B + \size{J,V(G')}_{G'} \le |B|+|G'|-2=|G|-1.\qedhere
	\]
\end{proof}

The next lemma has many different proofs \cite{erdos1979choosability,Entringer1985367,Hladky}.
Although it is known as Rubin's Block Lemma \cite{erdos1979choosability}, the lemma
was implicit in the much earlier work of Gallai \cite{gallai1963kritische} and Dirac.

\begin{lem}[Rubin's Block Lemma] If $G$ is a $2$-connected graph that is not complete
	and not an odd cycle, then $G$ contains an even cycle with at most one chord. \end{lem} 

\begin{thm}\label{micBasics} A connected graph $G=(V,E)$ is a Gallai tree if and
	only if $\mic(G)=|G|-1$; otherwise $\mic(G)\ge|G|$. \end{thm}
\begin{proof}
	By Lemmas \ref{mmic} and \ref{mGt-lem}, it suffices to show that $\mic(G)\geq|G|$
	when $G$ is not a Gallai tree. In this case Rubin's Block Lemma implies $G$ has
	an induced even cycle $C$ with only one possible chord $xy$. As $C$ has an equitable
	$2$-coloring, it has an independent set $A$ with $|A|=|C|/2$ that contains at
	most one of $x$ and $y$. Then $A$ is independent in $H:=G[C]$, so $\mic(H)\geq|H|$.
	By Lemma~\ref{mmic}, $\mic(G)\geq\mic(H)+|G-H|\geq|G|$. 
\end{proof}

\subsection{$f$-AT graphs}

For a graph $G$, we define $\func{d_{0}}{V(G)}{\IN}$ by $d_{0}(v)\DefinedAs d_{G}(v)$.
The $d_{0}$-choosable graphs were first characterized by Borodin \cite{borodin1977criterion}
and independently by Erd\H{o}s, Rubin and Taylor \cite{erdos1979choosability}. The
connected graphs which are not $d_{0}$-choosable are precisely the Gallai trees
(connected graphs in which every block is complete or an odd cycle). Hladk{\`{y}},
Kr{á}l and Schauz \cite{Hladky} generalized this classification to online $d_{0}$-choosable
graphs. In fact, they proved a classification in terms of Alon-Tarsi orientations
as follows. A subgraph $H$ of a directed multigraph $D$ is called \emph{Eulerian}
if $d_{H}^{-}(v)=d_{H}^{+}(v)$ for every $v\in V(H)$. We call $H$ \emph{even}
if $\size{H}$ is even and \emph{odd} otherwise. Let $EE(D)$ be the number of even,
spanning, Eulerian subgraphs of $D$ and $EO(D)$ the number of odd, spanning, Eulerian
subgraphs of $D$. Note that the edgeless subgraph of $D$ is even and hence we always
have $EE(D)>0$.

Let $G$ be a graph and $\func{f}{V(G)}{\IN}$. We say that $G$ is \emph{$f$-Alon-Tarsi}
(for brevity, \emph{$f$-AT}) if $G$ has an orientation $D$ where $f(v)\ge d_{D}^{+}(v)+1$
for all $v\in V(D)$ and $EE(D)\ne EO(D)$. One simple way to achieve $EE(D)\ne EO(D)$
is to have $D$ be acyclic since then we have $EE(D)=1$ and $EO(D)=0$. In this
case, ordering the vertices so that all edges point the same direction and coloring
greedily shows that $G$ is $f$-choosable. If we require $f$ to be constant, we
get the familiar \emph{coloring number} $\col(G)$; that is, $\col(G)$ is the smallest
$k$ for which $G$ has an acyclic orientation $D$ with $k\ge d_{D}^{+}(v)+1$ for
all $v\in V(D)$. Alon and Tarsi \cite{Alon1992125} generalized from the acyclic
case to arbitrary $f$-AT orientations.

\begin{lem}\label{AlonTarsi} If a graph $G$ is $f$-AT for $\func{f}{V(G)}{\IN}$,
	then $G$ is $f$-choosable. \end{lem}

\noindent Schauz \cite{schauz2010flexible} extended this result to online $f$-choosability.

\begin{lem}\label{Schauz} If a graph $G$ is $f$-AT for $\func{f}{V(G)}{\IN}$,
	then $G$ is online $f$-choosable. \end{lem}

Hladk{\`{y}}, Kr{á}l and Schauz \cite{Hladky} proved the following. \begin{thm}\label{ATDegreeChoosable}
	A connected graph is $d_{0}$-AT if and only if it is not a Gallai tree. \end{thm}

\subsection{$d_{0}$-KP graphs}

Acyclic orientations are also a special case of kernel-perfect orientations. A graph
$H$ is $f$-KP if $H$ has a kernel-perfect oriented supergraph $H'$ with vertex set $V(H)$, 
where $f(v)>d_{H'}^{+}(v)$ for all $v\in V(H')$. This supergraph for $f$-KP gives
us more power. For example $K_{4}-e$ has no kernel-perfect orientation $D$ with
$d_{G}(v)\geq d_{D}^{+}(v)+1$ for every vertex $v$, but if $H$ is the result of
doubling the edge that is in two triangles, then there is a kernel-perfect orientation
$D$ of $H$ with $d_{G}(v)\ge d_{D}^{+}(v)+1$. %
{} We could allow a supergraph for $f$-AT as well, but this doesn't give us any more
power. We will now prove the classification of connected $d_{0}$-KP graphs.

\begin{lem}\label{d0subgraphKP} A connected graph $G$ is $d_{0}$-KP if and only
	if some nonempty induced subgraph of $G$ is $d_{0}$-KP. \end{lem} 

\begin{proof} The `only if' direction is trivial. For the other direction, suppose
	$G$ has a nonempty induced subgraph that is $d_{0}$-KP and choose $H$ to be a
	maximal such subgraph. If $H=G$, we are done, so suppose not. Let $S=N_{G}(V(H))\sm V(H)$.
	Then $S\ne\emptyset$ since $G$ is connected. We show that $G[V(H)\cup S]$ is $d_{0}$-KP,
	violating maximality of $H$. Start with a kernel-perfect oriented supergraph $H'$
	of $H$ where $d_{H}(v)>d_{H'}^{+}(v)$ for all $v\in V(H')$. Now create an oriented
	supergraph $Q'$ of $Q\DefinedAs G[V(H)\cup S]$ by directing all edges from $V(H)$
	to $S$ into $S$ and replacing each edge in $Q[S]$ with arcs going both ways. Clearly,
	this oriented graph is kernel-perfect. Each vertex in $Q'[S]$ has at least one in-edge
	coming from $V(H)$, so we have $d_{Q}(v)>d_{Q'}^{+}(v)$ for all $v\in V(Q')$ as
	desired. \end{proof}

\begin{cor}\label{classifyd0} A connected graph is $d_{0}$-KP if and only if it
	is not a Gallai tree. \end{cor} \begin{proof} If $G$ is a Gallai tree, then it
	is not $d_{0}$-choosable and hence not $d_{0}$-KP by the Kernel Lemma. For the
	other direction, let $G$ be a connected graph that is not a Gallai tree. By Theorem
	\ref{micBasics}, we have $\mic(G)\ge|G|$. Applying Main Lemma with $f(v)\DefinedAs d_{G}(v)$
	for all $v\in V(G)$ gives a nonempty induced subgraph $H$ of $G$ that is $f_{H}$-KP
	where $f_{H}(v)\DefinedAs f(v)+d_{H}(v)-d_{G}(v)=d_{H}(v)$ for $v\in V(H)$. Now
	Lemma \ref{d0subgraphKP} shows that $G$ is $d_{0}$-KP. \end{proof}

\bibliographystyle{amsplain}
\bibliography{GraphColoring1}
 
\end{document}